\documentclass{article}
\usepackage{amssymb}
\usepackage{amsmath}
\usepackage{amsthm}
\usepackage[bookmarks]{hyperref}
\usepackage{tikz}
\usetikzlibrary{positioning}

\tikzset{>=stealth}

\author{Jonathan Breuer and Eyal Seelig \thanks{Institute of Mathematics, The Hebrew University of Jerusalem, Jerusalem, 91904, Israel.
Emails: jbreuer@math.huji.ac.il, eyal.seelig@mail.huji.ac.il}}
\title{Spectral Gaps for Jacobi Matrices on Graphs}

\numberwithin{equation}{section}

\theoremstyle{definition}

\theoremstyle{remark}
\newtheorem*{rem}{Remark}

\theoremstyle{plain}

\newtheorem{thm}{Theorem}[section]
\newtheorem{conjecture}{Conjecture}[section]

\newtheorem{prop}[thm]{Proposition}
\newtheorem{lem}[thm]{Lemma}
\newtheorem{cor}[thm]{Corollary}

\newcommand{\norm}[1]{\left\Vert #1\right\Vert }
\newcommand{\C}{{\mathbb{C}}}
\newcommand{\R}{{\mathbb{R}}}

\date{}
\begin{document}
\maketitle


\begin{center}{\it In celebration of the 80th birthday of Dany Leviatan}
\end{center}

\begin{abstract}
We study bounds on eigenvalue gaps for finite quotients of periodic Jacobi matrices on trees. We prove an Alon-Boppana type bound for the spectral gap and a comparison result for other eigenvalue gaps.
\end{abstract}

\section{Introduction}

Let $\mathcal G$ be a finite connected graph with vertex set $V$ and edge set $E$. We consider in this paper undirected edges and allow multiple edges between vertices, as well as loops. We do not allow vertices of degree $1$ (`leaves'). For two vertices, $v, w \in V$, we write $v \sim w$ (and refer to them as \emph{neighbors}) if there is an edge, $e \in E$, that connects them. In cases when there is a single edge between the vertices $v$ and $w$, we will denote that edge by $(v,w)$.

The spectral gap of $\mathcal G$, namely the gap between the highest and second highest eigenvalue of its adjacency matrix, has been a subject of considerable interest for several decades (\cite{bilu-linial,ceccherini,chung,cioaba,friedman,greenberg,gz,hoory,hlw,jiang,mss,mohar,Alon,srivastava-trevisan,wang-zhang,young} is a very partial list). The size of the spectral gap is of particular relevance in the study of networks and is known to be intimately connected to the geometry of the graph. Especially important is the study of graphs with a large spectral gap, known as expanders \cite{hlw}, where an important approach consists of studying graph covers (see, e.g., \cite{bilu-linial,bordenave-collins,mss}).

Perhaps the most fundamental result in this area is the Alon-Boppana bound \cite{Alon} (see also \cite{friedman}) bounding from above the size of the spectral gap. Originally formulated for regular graphs, it says that if $\lambda_2$ is the second highest eigenvalue of the adjacency matrix of a $d$-regular graph, $\mathcal G$, with diameter at least $2r$, then
\begin{equation} \label{eq:AB}
\lambda_2 \geq 2\sqrt{d-1}-\frac{2\sqrt{d-1}-1}{r}.
\end{equation}
Notably, $2\sqrt{d-1}$ is the spectral radius of the $d$-regular tree, which is the universal cover of any $d$-regular graph. Various generalizations and extensions of this result are discussed below. In this paper we want to discuss bounds of this type for a family of operators generalizing graph adjacency matrices, known as graph Jacobi matrices.

A \emph{Jacobi matrix} on a graph is an operator associated with a set of positive numbers, $\{a_e\}_{e \in E}$, assigned to the edges, and a set of real numbers, $\{b_v\}_{v \in V}$, assigned to the vertices (we refer to the set $\{b_v\}$ as `the potential'). Its matrix in the basis of vertex delta functions is given by the following formula
\begin{equation} \nonumber
J_{v,w}=\left\{ \begin{array}{cc} b_v+2\sum_{\textrm{loops at } v} a_e &\quad \textrm{ if } v=w \textrm{ and the sum is over loops at } v \\
\sum_{e \textrm{ between } v \textrm{ and } w} a_e &\quad \textrm{ if } v\sim w \textrm{ and the sum is over all edges between them}\\
0 & \quad \textrm{if } v\neq w \textrm{ are not neighbors}  \\
\end{array}
\right.
\end{equation}

The universal cover of a finite, leafless, (regular) graph is an infinite (regular) tree. The lift of a finite graph Jacobi matrix to the graph's universal cover is known as a \emph{periodic Jacobi matrix on a tree} \cite{abs}. Such objects have attracted quite a bit of attention in recent years \cite{aomoto1,aomoto2,abs,abks,bbgvss,bgvm,figa-talamanca-steger,vargas,klw1,klw2,sunada,woess}.

While most of the work in this context has concentrated on the spectral properties of the periodic lifts, properties of \emph{finite} lifts of a graph Jacobi matrix (and their connection to the universal cover) have also received some attention \cite{abks,bgvm,csz}. We note, in particular, \cite{abks}, which studies the convergence of the normalized eigenvalue counting measure along towers of finite lifts. The current paper is written in this spirit. Its main purpose is to prove lower bounds of the type \eqref{eq:AB} for finite quotients of periodic Jacobi matrices on trees.

As remarked above, extensions of the Alon-Boppana bound have been considered in several contexts. Greenberg \cite{greenberg} studied adjacency matrices on non-regular graphs. He showed that if $\mathcal{G}_n$ is a sequence of finite graphs with a common universal cover, $T$, such that the number of vertices, $|V(\mathcal{G}_n)| \to \infty$, then $\liminf\lambda_2 (\mathcal{G}_n) \geq \rho(T)$ where $\rho(T)$ is the spectral radius of $T$ (see also \cite{cioaba,gz,mohar} for related results and \cite{ceccherini} for an extension to graphs with edge weights).

A different bound for the case of non regular graphs (with no edge weights) was given by Hoory, using the average degree of the graph \cite{hoory}. In `robust' cases, Hoory's bound also involves a bound on the rate of convergence. The papers \cite{chung,jiang,wang-zhang,young} contain various extensions of Hoory's result, allowing edge weights, and improving both the lower bound and the convergence rate. Srivastava-Trevisan in \cite{srivastava-trevisan} also allow for general edge weights, in the slightly different context of proving concentration of the spectrum of the Laplacian. Recently, Bordenave-Collins \cite{bordenave-collins} obtained an Alon-Boppana type lower bound for the operator norm of non-commutative polynomials in permutation matrices with coefficients from a $C^*$ algebra (which generalize adjacency matrices on graph covers).

We note that none of the results mentioned above, except for \cite{bordenave-collins}, allow for a vertex potential.

\smallskip

Given a Jacobi matrix, $J_0$ on a finite graph, $\mathcal G_0$, our aim in this paper is to formulate lower bounds on the second eigenvalue of lifts of $J_0$ to finite coverings of $\mathcal G_0$.
We start with noting that Greenberg's proof extends almost verbatim to the general case of Jacobi matrices. Thus

\begin{thm}\label{thm-greenberg-jacobi}
	Let $J_0$ be a Jacobi matrix
	on a finite graph $\mathcal G_0$.
	Let $\mathcal{G}_n$ be a sequence of finite covers of $\mathcal{G}_0$
	such that $|V(\mathcal{G}_n)| \to \infty$,
	and let $J_n$ be
	the lift of $J_0$ to $\mathcal{G}_n$.
	Let $J_T$ be the lift of $J_0$ to the
	universal cover $T$ and let $\sigma(J_T)$ denote its spectrum.
	Then
	\begin{equation*}
		\liminf_{n\to\infty}
		\lambda_2(J_n)
		\geq
		\sup\sigma(J_T).
	\end{equation*}
\end{thm}

For the sake of completeness, the proof is given in Appendix \ref{GreenbergProof}.
Note that the proof actually implies the following: let $(\mathcal{G}_n,J_n)$ be a sequence of finite lifts of $(\mathcal G_0,J_0)$ and let $\nu_n$ be the normalized eigenvalue counting measure of $J_n$. Assume that $|V(\mathcal{G}_n)|\rightarrow \infty$ and that $\nu_n$ converges to a measure $\nu_\infty$. Then
\begin{equation} \label{eq:Greenberg}
\sup \textrm{supp} \nu_\infty \geq \sup \sigma(J_T).
\end{equation}
This of course implies Theorem \ref{thm-greenberg-jacobi}, in fact with $\lambda_2$ replaced by $\lambda_m$ for any $m$.
Note that, by \cite{abks}, there is a limiting eigenvalue counting measure whose support is $\sigma(J_T)$ (i.e.\ the \emph{density of states} of $J_T$). Thus, $\sigma(J_T)$ has the smallest supremum out of all supports for limiting eigenvalue counting measures. A natural question is whether this is true in terms of containment of sets. Clearly, if the covers are not connected it is easy to construct counterexamples (if $\mathcal{G}_n$ is a union of $n$ copies of $\mathcal{G}_0$ then $\nu_\infty$ is supported on a discrete set, whereas $\sigma(J_T)$ is not discrete). We conjecture that this is the only reason this could fail.
\begin{conjecture} \label{conj:GeneralizedAB}
Let $J_0$ be a Jacobi matrix on a finite connected graph $\mathcal G_0$. Let $\mathcal{G}_n$ be a sequence of finite connected covers of $\mathcal G_0$ such that $|V(\mathcal{G}_n)| \to \infty$, and let $J_n$ be the lift of $J_0$ to $\mathcal{G}_n$. Let $J_T$ be the lift of $J$ to the universal cover $T$.

Let $\nu_n$ be the normalized eigenvalue counting measure of $J_n$ and assume that the weak limit $\nu_\infty=\lim_{n \to \infty}\nu_n$ exists. Then
\begin{equation} \nonumber
\sigma(J_T)\subseteq \textrm{supp} \nu_\infty.
\end{equation}
\end{conjecture}
Note that in general, $\sigma(J_T)$ is not connected, and there is no reason to expect the support of $\nu_\infty$ to be connected. The results of \cite{bordenave-collins} show that for any point in $\sigma(J_T)$, one can find nearby eigenvalues of lifts to finite covers, provided these finite covers are large enough and locally tree-like around two points. This is not enough to prove Conjecture \ref{conj:GeneralizedAB} but is an indication in its favor.

\smallskip

The main issue with asymptotic results such as Theorem \ref{thm-greenberg-jacobi} is the absence of error estimates. Our main new result in this paper is an Alon-Boppana type bound, with an explicit error estimate, for finite quotients of periodic Jacobi matrices on trees (or, equivalently, finite lifts of a Jacobi matrix on a finite graph). There are no restrictions in our bounds on the diagonal (the $\{b_v\}_{v\in V}$) or on the `edge-weights' (the $\{a_e\}_{e\in E}$), aside from the latter being positive.

In order to present this bound, we need to introduce some notation, recalling and extending the `lego block picture' from \cite{abs} (also see \cite{sunada}). This picture of a cover of $\mathcal{G}_0$ consists of choosing a fundamental domain and tracking its copies and their connections with each other. To be more explicit (and graphic), fix a finite connected graph, $\mathcal{G}_0$, with $p$ vertices. Let $\ell$ be the number of independent closed paths in $\mathcal{G}_0$ (aka the first Betti number of $\mathcal{G}_0$). It is possible to find $\ell$ edges such that by removing them $\mathcal{G}_0$ becomes a tree. It is important to note that the edges to be removed are \emph{not} determined uniquely by this requirement; there is a choice involved in this procedure. Denote the chosen edges by $\{e_1,\ldots,e_\ell\}$. To obtain a `lego block' we do not remove these edges, but rather cut them in half, obtaining $2\ell=d$ `dangling' half-edges $e_1^+,\ldots,e_\ell^+,e_1^-,\ldots,e_\ell^-$. An $n$-cover of $\mathcal{G}_0$ can be obtained by taking a $d$-regular graph, $G$, on $n$ edges, placing copies of the lego block at each vertex and gluing an $e_j^+$ to an $e_j^-$ along an edge. It is not hard to see that any cover of $\mathcal{G}_0$ can be obtained by this procedure.

Now, if $J_0$ is a Jacobi matrix on $\mathcal{G}_0$, let $B$ be its restriction to the tree obtained from $\mathcal{G}_0$ by removing $e_1,\ldots,e_\ell$, and write
\begin{equation} \label{eq:LegoJacobi1}
J=B+\sum_{i=1}^\ell (A_i+A_i^*)
\end{equation}
where $A_i$ is a $p\times p$ matrix defined by choosing a direction for $e_i$ and letting
$$(A_i)_{vw}=a_{e_i}$$
if $v$ is the initial vertex of $e_i$ and $w$ is the terminus of $e_i$, and $0$ everywhere else (so if $e_i$ is a loop at $v$ then $A_i$ has $a_{e_i}$ on the diagonal entry corresponding to $v$ and $A_i=A_i^*$). The lift, $J$, of $J_0$ to a (finite or infinite) cover $\mathcal{G}$ of $\mathcal{G}_0$ can be regarded as an operator on $L^2(G, \C^p)$, where $G$ is a (finite or infinite) $d$-regular graph, and
\begin{equation} \label{eq:LegoJacobi2}
		(Jx)_u
		=
		B x_u
		+ \sum_{i=1}^\ell
		\left(
			A_i x_{u+e_i}
			+
			A_i^* x_{u-e_i}
		\right),
\end{equation}
where we slightly abuse notation and write $\left\{ u \pm e_i \right\}_{i=1}^\ell$ to denote the $d$ neighbors of $u$ in $G$. We regard $G$ as the `skeleton' of $\mathcal{G}$.
We let $A = \sum_{i=1}^\ell \left( A_i + A_i^* \right)$. Our main result is

\begin{thm} \label{thm-alon-boppana-jacobi1}
Let $J_0$ be a Jacobi matrix on a finite graph $\mathcal G_0$. Let $J$ be the lift of $J_0$ to a finite cover of $\mathcal G_0$, and we view $J$ as acting on $L^2(G,\C^p)$ for some $d$-regular finite graph $G$ as in \eqref{eq:LegoJacobi2}.

Suppose that there exist a pair of edges $e_1, e_2 \in E(G)$ weighted by either $A_j$ or $A_j^*$ for some fixed $j \in \{1,..., \ell\}$, such that the distance between the ends of $e_1$ and the ends of $e_2$ is greater than $2r+2$.

Then for every unit vector with non-negative entries, $y \in \C^p$,
\begin{equation}\label{ineq-explicit-asymptotics}
		\lambda_2(J)
		\geq
		\langle By,y \rangle
		+
		2 \sqrt{d - 1}
		\frac	
		{\langle Ay, y \rangle}
		{d}
		- \frac{C_j}{r}
\end{equation}
	where
\begin{equation*}
		C_j=
		(2 \sqrt{d - 1} - 1)
		\frac{ \langle  Ay, y \rangle}{d}
		+
		\left| \frac{ \langle Ay, y \rangle}{d}
		- \langle A_j y,  y \rangle \right|.
\end{equation*}
\end{thm}

\begin{rem}
Since, as noted above, the choice of the edges $e_1,\ldots,e_\ell$ is not unique, it follows that there could be several different choices for the matrices $A$ and $B$ above. The theorem holds for each choice satisfying the conditions so an optimization over both this choice and the vector $y$ can be carried out.
\end{rem}

\begin{rem}
	If the diameter of $G$ is larger than $2r+4$, then
	the largeness condition holds for some choice of $j$.
	Indeed, let $v,w$ be a pair of vertices
	of distance larger than $2r+4$.
	Let $(v, v_1,, ..., w)$ be the shortest path between them.
	Choose $u$ a neighbor of $w$
	such that $A_{(w,u)}
	\in \{A_{(v,v_1)}, A^*_{(v,v_1)} \}$.
	Then the edges $(w,u), (v,v_1)$ satisfy the condition.

It follows that for any sequence of covers, $\mathcal{G}_n$, with $V(\mathcal{G}_n)\rightarrow \infty$, if $J_n$ is the lift of $J_0$ to $\mathcal{G}_n$ then for any unit vector with non-negative entries $y\in \C^p$
\begin{equation*}
		\underset{n\to\infty}{\liminf}
		\lambda_2(J_n)
		\geq
		\langle By,y \rangle
		+
		2 \sqrt{d - 1}
		\frac	
		{\langle Ay, y \rangle}
		{d}.
	\end{equation*}
This should be compared with Theorem \ref{thm-greenberg-jacobi}. In fact, we show in Proposition \ref{thm-spectral-bound} below that the right hand side in this inequality is a lower bound for the spectral radius of $J_T$.
\end{rem}

\begin{cor}[Bound on the spectral gap]
With the notation and under the conditions of Theorem \ref{thm-alon-boppana-jacobi1}, let $y$ be the normalized Perron-Frobenius eigenvector of $J_0$. Then
	\begin{equation*}
		\lambda_1(J) - \lambda_2(J)
		\leq
		\left( d - 2 \sqrt{d-1} \right)
		\frac	
		{\langle Ay, y \rangle}
		{d}
		+ \frac{C_j}{r}.
	\end{equation*}
\end{cor}
\begin{proof}
	Plugging $B = J_0 - A$ into
	inequality \eqref{ineq-explicit-asymptotics}, one gets
	\begin{equation*}
		\lambda_2(J) \geq
		\langle J_0 y, y \rangle
		- \langle A y, y \rangle
		+ 2 \sqrt{d-1}
		\frac{
		\langle A y, y \rangle}
		{d}
		- \frac{C_j}{r}.
	\end{equation*}
	Picking $y$ to be the normalized Perron-Frobenius
	eigenvector of $J_0$,
	and using Lemma \ref{lem-lift-spec-radius},
	\begin{equation*}
		 \langle J_0 y, y \rangle = \lambda_1(J_0) = \lambda_1(J).
	\end{equation*}
	Rearranging the terms in the inequality above finishes the proof.
\end{proof}

\begin{rem}
Note that in the case that $J$ is the adjacency matrix of a $d$-regular graph, we can take $\mathcal{G}_0$ to be a single vertex with $\ell=d/2$ loops (recall $d$ here is even). In this case $B=0$ so that $\langle Ay,y \rangle=d$ and we recover the standard Alon-Boppana bound. Of course, our bound, as stated here, does not treat the case of odd $d$ (this can be fixed using the notion of `half-loops' \cite{friedman-halfloops}).
\end{rem}

The rest of this paper is structured as follows. The next section describes the proof of Theorem \ref{thm-alon-boppana-jacobi1} and Section 3 collects some results on the comparison of eigenvalue bounds between different Jacobi matrices on the same graph. The appendix collects some auxiliary results.

\textbf{Acknowledgments.}
Research supported in part by the Israel Science Foundation (Grant No.\ 1378/20) and in part by the United States-Israel Binational Science Foundation (Grant No.\ 2020027).


\section{An Alon-Boppana Bound for Lifts of Jacobi Matrices}

We begin with a technical lemma, whose proof is heavily inspired by the proof in \cite{Alon}.

\begin{lem}\label{lem-technical}
	Let $J_0$ be a Jacobi matrix on a finite connected graph $\mathcal{G}_0$, with $p$ vertices. Let $\ell$ be its first Betti number and $\{e_1,\ldots,e_\ell\}$ be a choice of edges whose removal turns $\mathcal{G}_0$ to a tree.
	Let $J_T$ be its lift to a periodic Jacobi matrix
	on the universal cover of $\mathcal{G}_0$, and we view $J_T$ as acting on the space of $\mathbb{C}^p$ valued functions on the $d=2\ell$ regular tree, $\mathcal T_d$.
	Choose $1\leq i_0 \leq \ell$ and fix an edge $(v,v') \in E(\mathcal T_d)$
	with weight $A_{i_0}$ or $A_{i_0}^*$ (which we henceforth denote by $A_\star$).
	
For every
	$r > 0$
	and any unit vector with non-negative entries, $y \in \R^p$,
	there exists a unit vector $x \in L^2(\mathcal T_d, \R^p)$
	supported on a ball of radius $r-1$ around
	the pair $\{v,v'\}$, such that
	\begin{equation*}
		\langle J_T x, x \rangle
		\geq
		\langle By,y \rangle
		+
		2 \sqrt{d - 1}
		\frac	
		{\langle Ay, y \rangle}
		{d}
		- \frac{C_\star}{r}
	\end{equation*}
	where
	\begin{equation*}
		C_\star=
		(2 \sqrt{d - 1} - 1)
		\frac{ \langle  Ay, y \rangle}{d}
		+
		\left| \frac{ \langle Ay, y \rangle}{d}
		- \langle A_\star y,  y \rangle \right|.
	\end{equation*}
\end{lem}

\begin{proof}
	Denote the pair ${V_0 = \{v, v'\}}$,
	and let $V_k$ be the set of vertices in $\mathcal T_d$
	of distance $k$ from $V_0$.
	Define
	$x \in L^2(\mathcal T_d, \C^p)$ by

	$$x_u = \begin{cases}
		\frac{1}{\sqrt{2r}}
		\left( \frac{1}{\sqrt{d - 1}} \right)^k
		\cdot y & \text{if } u\in V_k \text{ and } k < r, \\
		0 & \text{otherwise}.
	\end{cases}$$
	Since $y$ is normalized and $|V_k| = 2(d-1)^k$,
	it is clear that
	$$\norm{x}^2 = \frac{1}{2r} \sum_{k=0}^{r-1} |V_k|
	\left( \frac{1}{\sqrt{d - 1}} \right)^{2k} = 1.$$

	Recall the notation $A = \sum_{i=1}^\ell \left( A_i + A_i^* \right)$.
	This is a $p \times p$ matrix.
	Define the operator $\mathcal A$ on $L^2(\mathcal{T}_d, \mathbb{C}^p)$
	by putting copies of $A$ on the diagonal, namely
	\begin{equation*}
		(\mathcal A x)_u = A x_u.
	\end{equation*}
	Define the operator
	$L = \mathcal A - J_T$, so
	\begin{equation*}
		(Lx)_u =
		-B x_u + \sum_{i=1}^\ell
		\left( A_{i} (x_u - x_{u+e_i})
		+ A_{i}^* (x_u - x_{u-e_i}) \right)
	\end{equation*}
	and
	\begin{align*}
		\langle Lx, x \rangle
		&=
		- \langle B y, y \rangle
		+ \sum_{u \in V(\mathcal T_d)}
		\sum_{i=1}^\ell \left(
		\langle A_i(x_u - x_{u+e_i}), x_u \rangle +
		\langle A_i^*(x_u - x_{u-e_i}), x_u \rangle
		\right).
	\end{align*}
	Note that every edge $(u,w)\in E(\mathcal T_d)$
	appears exactly twice in the sum.
	If $A_{(u,w)}$ denotes the weight of
	the edge (which would be either $A_i$ or $A_i^*$
	for some $i \in \{1,...,\ell\}$),
	the edge appears
	\begin{itemize}
		\item once as $ \langle A_{(u,w)} (x_u - x_w), x_u \rangle$,
		\item and once as $ \langle A_{(u,w)}^* (x_w - x_u), x_w \rangle$.
	\end{itemize}
	Their sum is
	$ \langle x_u - x_w, A_{(u,w)}^* x_u - A_{(u,w)} x_w \rangle$,
	so
	\begin{equation*}
		\langle Lx, x \rangle
		=
		-\langle B y, y \rangle
		+ \sum_{(u, w) \in E(\mathcal T_d)}
		 \langle x_u - x_w, A_{(u,w)}^* x_u - A_{(u,w)} x_w \rangle.
	\end{equation*}
	We shall enumerate the edges in the sum
	according to their distance from $V_0$.
	\begin{equation}\label{eq-sumedges}
		\langle Lx, x \rangle
		= -\langle By, y \rangle
		+ \sum_{k=0}^{r-1}
		\sum_{(u, w) \in E_k}
		 \langle x_u - x_w, A_{(u,w)}^* x_u - A_{(u,w)} x_w \rangle
	\end{equation}
	where $E_k$ denotes the set of edges connecting $u \in V_k$
	to $w \in V_{k+1}$.

	If $(u, w) \in E_k$ for $k<r-1$, then
	\begin{align*}
		\langle x_u - x_w, A_{(u,w)}^* x_u - A_{(u,w)} x_w \rangle
		&=
		\langle x_u - x_w, A_{(u,w)}^* x_u \rangle
		-
		\langle x_u - x_w, A_{(u,w)} x_w \rangle \\
		&=
		\frac{1}{2r}
		\left(
			\left( \frac{1}{\sqrt{d -1}} \right)^k -
			\left( \frac{1}{\sqrt{d -1}} \right)^{k+1}
		\right)  \\
		&\quad \times
		\left(
			\left( \frac{1}{\sqrt{d -1}} \right)^k
			\langle y, A_{(u,w)}^* y \rangle
			-
			\left( \frac{1}{\sqrt{d -1}} \right)^{k+1}
			\langle y, A_{(u,w)} y \rangle
		\right) \\
		&=
		\frac{1}{2r}
		\left(
			\left( \frac{1}{\sqrt{d -1}} \right)^k -
			\left( \frac{1}{\sqrt{d -1}} \right)^{k+1}
		\right)^2
		\langle A_{(u,w)} y, y \rangle \\
		&=
		\frac{1}{2r}
		\frac{d - 2 \sqrt{d -1}}{d - 1}
		\left( \frac{1}{\sqrt{d -1}} \right)^{2k}
		\langle A_{(u,w)} y, y \rangle .
	\end{align*}
	And if $k=r-1$ then
	\begin{equation*}
		\langle x_u - x_w, A_{(u,w)}^* x_u - A_{(u,w)} x_w \rangle
		=
		\frac{1}{2r}
		\left( \frac{1}{\sqrt{d - 1}} \right)^{2(r-1)}
		\langle A_{(u,w)} y,y \rangle.
	\end{equation*}
	We break down the sum in \eqref{eq-sumedges} even further.
	For $k=1,...,r-1$ and $j=1,...,\ell$,
	let
	$V_{k,j} = \left\{ u \in V_k \mid
		u + e_j \in V_{k-1} \text{ or } u - e_j \in V_{k-1} \right\}$
	be the set of vertices in $V_k$ connected to $V_{k-1}$
	via either $A_j$ or $A_j^*$.
	More importantly, we can easily enumerate
	the $2\ell - 1$ edges going from $V_{k,j}$ to $V_{k+1}$:
	for every vertex, there will be
	one edge weighted by $A_j$ (or its conjugate)
	and two edges weighted by $A_i$ (and its conjugate)
	for every $i \neq j$.
	We define the sets $V_{0,j}$ in a compatible manner:
	Recall that the function $x$ is defined radially around an edge
	whose weight we denote by $A_\star$.
	Let
	\begin{equation*}
		V_{0, j} =
		\begin{cases}
			V_0  & \text{if }A_\star\text{ is either }A_j\text{ or }A_j^* \\
			\emptyset  & \text{otherwise}.
		\end{cases}
	\end{equation*}
	Now,
	\begin{equation*}
		\langle Lx, x \rangle
		=
		- \langle By, y \rangle
		+ \sum_{k=0}^{r-1}
		\sum_{j=1}^\ell
		\sum_{(u, w) \in E_{k,j}}
		 \langle x_u - x_w, A_{(u,w)}^* x_u - A_{(u,w)} x_w \rangle.
	\end{equation*}
	where
	$E_{k,j} := \{(u, w) \mid
		u \in V_{k,j} ,\enskip w \in V_{k+1}\}$.
	Using the simplification of the inner product calculated
	above, we get for $k<r-1$
	\begin{align*}
		\sum_{j=1}^\ell
		\sum_{(u, w) \in E_{k,j}}
		 \langle x_u - x_w, A_{(u,w)}^* x_u - A_{(u,w)} x_w \rangle
		&=
		\frac{1}{2r}
		\frac{d - 2 \sqrt{d - 1}}
		{d - 1}
		\left( \frac{1}{\sqrt{d - 1}} \right)^{2k} \\
		&\quad \times
		\sum_{j=1}^\ell
		\left| V_{k, j} \right|
		\left( \langle A_j y,y \rangle
		+ 2\sum_{i \neq j} \langle A_i y,y \rangle\right)
	\end{align*}
	and similarly for $k = r-1$
	\begin{align*}
		\sum_{j=1}^\ell
		\sum_{(u, w) \in E_{r-1,j}}
		 \langle x_u - x_w, A_{(u,w)}^* x_u - A_{(u,w)} x_w \rangle
		 &=
		 \frac{1}{2r}
		\left( \frac{1}{\sqrt{d - 1}} \right)^{2(r-1)} \\
		&\quad \times
		\sum_{j=1}^\ell
		\left| V_{r-1, j} \right|
		\left( \langle A_j y,y \rangle
		+ 2\sum_{i \neq j} \langle A_i y,y \rangle\right).
	\end{align*}
	It will be useful to explicitly calculate $|V_{k,j}|$.
	For every fixed $j$, these sizes obey a recursion formula
	due to the tree structure
	\begin{equation}\label{eq-sizesrecursion}
		|V_{k,j}|
		=
		2 |V_{k-1}| - |V_{k-1,j}|
	\end{equation}
	with initial conditions
	\begin{equation*}
		|V_{0, j}| =
		\begin{cases}
			2  & \text{if }A_\star\text{ is either }A_j\text{ or }A_j^* \\
			0  & \text{otherwise}.
		\end{cases}
	\end{equation*}
	Solving the recursion equation, one gets
	\begin{equation*}
		|V_{k, j}| =
		\frac{2|V_k|}{d}
		+ (-1)^{k+1} \cdot \beta_j
	\end{equation*}
	where
	\begin{equation*}
		\beta_j =
		\begin{cases}
			\frac{4}{d} - 2 &
			\text{if }A_\star\text{ is either }A_j\text{ or }A_j^* \\
			\frac{4}{d} & \text{otherwise}.
		\end{cases}
	\end{equation*}
	Now an explicit calculation shows
\begin{equation} \nonumber
\sum_{j=1}^\ell
		|V_{k, j}|
		\left( \left<A_j y,y \right>
		+
		2 \sum _{i \neq j}
		\left<A_i y,y \right>
		\right)
=|V_k| (d-1) \frac{\langle Ay,y \rangle}{d}
+(-1)^{k} \cdot 2
\left(\frac{\langle Ay,y \rangle}{d} 
- \langle A_\star y,y \rangle \right).
\end{equation}
	Substituting in the equations above, we get
	\begin{align*}
		\left< Lx, x \right>
		& =
		-\left< By, y \right> \\
		&\quad +
		\frac{1}{2r}
		(d - 2 \sqrt{d - 1})
		\frac{ \langle A y,y \rangle}{d}
		\sum_{k=0}^{r - 1}
		\left| V_k \right|
		\left( \frac{1}{\sqrt{d - 1}} \right)^{2k}
		 \\
		&\quad +
		\frac{1}{2r}
		(2 \sqrt{d - 1} - 1)
		\frac{ \langle A y,y \rangle}{d}
		|V_{r-1}|
		\left( \frac{1}{\sqrt{d - 1}} \right)^{2(r-1)}
		\\
		&\quad +
		\frac{1}{r}
		\cdot
		\frac{d - 2\sqrt{d - 1}}{d - 1}
		 \left(
		\frac{ \langle A y,y \rangle}{d}
		- \left<A_\star y,y \right> \right)
		\sum_{k=0}^{r-1}
		(-1)^k
		\left(\frac{1}{\sqrt{d - 1}} \right)^{2k} \\
		&\quad +
		\frac{1}{r}
		\cdot
		\frac{2 \sqrt{d-1} - 1}{d - 1}
		 \left(
		\frac{ \langle A y,y \rangle}{d}
		- \left<A_\star y,y \right> \right)
		(-1)^{r-1}
		\left(\frac{1}{\sqrt{d - 1}} \right)^{2(r-1)}.
	\end{align*}
	Let us treat each of the five terms separately.
	The first term is left as is.
	The second term clearly equals
	\begin{equation*}
		(d - 2 \sqrt{d - 1})
		\frac{ \langle Ay ,y \rangle}{d}.
	\end{equation*}
	Due to the tree structure
	$|V_k| = 2 (d - 1)^k$
	for every $k$,
	hence the third term equals
	\begin{equation*}
		\frac{1}{r}
		(2 \sqrt{d - 1} - 1)
		\frac{ \langle Ay, y\rangle}{d}.
	\end{equation*}
	The fourth term is bounded from above by
	\begin{equation*}
		\frac{1}{r} \cdot
		\frac{d - 2\sqrt{d - 1}}{d - 1}
		\left| \frac{ \langle Ay,y \rangle}{d} -
		\langle A_\star y,y \rangle \right|
	\end{equation*}
	and the fifth by
	\begin{equation*}
		\frac{1}{r} \cdot
		\frac{2 \sqrt{d-1} - 1}{d-1}
		 \left|
		\frac{ \langle Ay,y \rangle}{d} -
		\langle A_\star y,y \rangle \right|.
	\end{equation*}
	Therefore
	\begin{align*}
		\left<Lx,x \right>
		\leq
		-\left< By, y \right>
		+
		(d - 2 \sqrt{d - 1})
		\frac{ \langle Ay, y \rangle}{d}
		+
		\frac{C_\star}{r}
	\end{align*}
	where
	\begin{equation*}
		C_\star=
		(2 \sqrt{d - 1} - 1)
		\frac{ \langle  Ay, y \rangle}{d}
		+
		\left| \frac{ \langle Ay, y \rangle}{d}
		- \langle A_\star y,y \rangle \right|.
	\end{equation*}
	All that is left is to return to our Jacobi operator.
	Recall that $L = \mathcal A - J_T$
	and note that $ \langle \mathcal A x, x \rangle
	= \langle Ay, y \rangle$,
	so
	\begin{equation*}
		\langle J_T x, x \rangle
		\geq
		\langle By, y \rangle
		+ 2 \sqrt{d - 1}
		\frac
		{\langle Ay, y \rangle}
		{d}
		- \frac{C_\star}{r}.
	\end{equation*}
\end{proof}

\begin{proof}[Proof of Theorem \ref{thm-alon-boppana-jacobi1}]
	First, we note that by adding a constant diagonal term (which adds the same constant to both sides of \eqref{ineq-explicit-asymptotics}) we may assume that $J$ is positive definite.

	Second, for convenience we denote by $\widehat C$
	the R.H.S.~of Inequality \eqref{ineq-explicit-asymptotics}.
	For any edge $e$ in a graph,
	denote by $B(e,k)$
	the subgraph induced by all the vertices
	of distance at most $k$ from any
	of the ends of $e$.
	If $M$ is a matrix on a graph containing $e$,
	it can be restricted to $M \vert_{B(e,k)}$,
	which we think of as operating on the subspace $L^2(B(e,k))$
	by setting all edge weights outside of $B(e,k)$
	to zero before applying $M$.
	This is equivalent to $P \circ M \circ P$
	where $P$ is the projection
	from the entire graph to $B(e,k)$.

	Let $e_1, e_2 \in E(G)$ be as in the assumption.
	For $i=1,2$,
	denote by $\widetilde e_i$
	an arbitrary element in the fiber of $e_i$
	in the universal covering tree.
	Let $J_T$ be the lift of $J_0$
	to that tree, so by
	Lemma \ref{lem-technical},
	there exist unit vectors $x_i \in B(\widetilde e_i, r)$,
	vanishing on the boundary of their respective subgraphs,
	such that
	\begin{equation*}
		\langle J_T \vert_{B(\widetilde e_i, r)}
			x_i, x_i \rangle=\langle J_T x_i,x_i \rangle
		\geq
		\widehat C.
	\end{equation*}
	It follows that
	\begin{equation*}
		\lambda_1(J_T \vert_{B(\widetilde e_i, r)})
		\geq \widehat C.
	\end{equation*}
	Composition with the covering map
	induces an injection of closed paths
	in $B(\widetilde e_i, r)$ to closed paths in
	$B(e_i, r)$.

	Now note that for any positive definite Jacobi matrix, $J$, on a connected graph, its spectral radius is equal to $$\limsup_{n\to \infty} \left(\langle \delta_v, J^n \delta_v \rangle \right)^{1/n},$$
for any vertex $v$. Thus (again, using the positivity of the entries), we get
	\begin{equation*}
		\lambda_1(J \vert_{B(e_i, r)})
		\geq
		\lambda_1(J_T \vert_{B(\widetilde e_i, r)})
		\geq \widehat C.
	\end{equation*}
	Let $z_i$ denote the Perron-Frobenius eigenvector
	of $J \vert_{B(e_i, r)}$,
	and define the vector $w_i : V(G) \to \R_+$
	by
	\begin{equation*}
		w_i(v) =
		\begin{cases}
			z_i(v) & v \in B(e_i, r)  \\
			0 & \text{otherwise}.
		\end{cases}
	\end{equation*}
	It is not hard to see that
	\begin{equation*}
		\frac{ \langle J w_i, w_i \rangle}
		{\norm{w_i}^2}
		=
		\frac{ \langle J \vert_{B(e_i, r)} z_i, z_i \rangle}
		{\norm{z_i}^2}
		\geq
		\widehat C.
	\end{equation*}
	This bound on the Rayleigh quotient
	also holds for every linear combination
	of $w_1, w_2$.
	Indeed, we picked $e_1, e_2$ far apart enough
	so that the supports of $w_1, w_2$ are disjoint
	even after applying $J$ on them, so
	for every $\alpha, \beta \in \C$,
	\begin{equation*}
		\frac
		{ \langle J \left( \alpha w_1 + \beta w_2 \right),
		\alpha w_1 + \beta w_2 \rangle}
		{\norm{\alpha w_1 + \beta w_2}^2}
		=
		\frac
		{\alpha^2 \langle Jw_1, w_1 \rangle +
		\beta^2 \langle J w_2, w_2 \rangle}
		{\alpha^2 \norm{w_1}^2 + \beta^2 \norm{w_2}^2}
		\geq
		\widehat C.
	\end{equation*}
	To conclude,
	the min-max theorem tells us
	that the second largest eigenvalue is
	bounded from below by the maximal Rayleigh
	quotient of vectors perpendicular to the
	Perron-Frobenius eigenvector.
	Taking a linear combination
	$\alpha w_1 + \beta w_2$ perpendicular
	to the Perron-Frobenius eigenvector
	of $J$, we get
	\begin{equation*}
		\lambda_2(J) \geq \widehat C.
	\end{equation*}
	
\end{proof}

The following lower bound on the $\sup$ of the spectrum of a periodic Jacobi matrix on a tree, is a direct consequence of Lemma \ref{lem-technical}.

\begin{prop}\label{thm-spectral-bound}
	Let $J_T$ be a periodic Jacobi matrix
	as in Lemma \ref{lem-technical}.
	Then for every unit vector with non-negative entries, $y \in \C^p$,
	\begin{equation*}
		\sup \sigma \left( J_T \right)
		\geq
		\langle By,y \rangle
		+
		2 \sqrt{d - 1}
		\frac	
		{\langle Ay, y \rangle}
		{d}.
	\end{equation*}
\end{prop}
\begin{proof}
	This follows from Lemma \ref{lem-technical}
	and the fact that since $J_T$ is self-adjoint,
	\begin{equation*}
		\sup \sigma(J_T)
		=
		\underset{\norm{x} = 1}{\sup}
		\langle J_T x,x \rangle.
	\end{equation*}
\end{proof}
\begin{rem}
It is natural to ask whether
\begin{equation*}
\sup \sigma \left( J_T \right)
		= \sup \left(
		\langle By,y \rangle
		+
		2 \sqrt{d - 1}
		\frac	
		{\langle Ay, y \rangle}
		{d}\right)
\end{equation*}
where the right-hand side $\sup$ is taken over all possible choices of spanning trees for $\mathcal{G}_0$ and non-negative, normalized vectors $y$. Unfortunately, we were not able to prove or disprove this. An explicit variational formula for the spectral radius of $J_T$ was obtained by Garza-Vargas and Kulkarni in \cite{vargas}.
\end{rem}


\section{Comparison Bounds}

In the paper \cite{csz} Christiansen, Simon, and Zinchenko prove a comparison result for the gap between the top eigenvalue of a Jacobi matrix on a graph and the spectral radius of its lift to the universal cover. They use the following `ground state representation' for the quadratic form of a Jacobi matrix on a graph.

\begin{lem}[essentially Theorem 2.1 in \cite{csz}] \label{thm:CSZ}
 Let $J$ be a Jacobi matrix on a finite graph, $\mathcal{G}$. Suppose that $\psi$ is a positive function on $\mathcal{G}$ satisfying
 $$
 J\psi=\lambda_1\psi
 $$
$($where $\lambda_1$ is the top eigenvalue of $J)$. Then for any $f: \mathcal{G}\to \mathbb{C}$
\begin{equation} \label{eq:GroundState}
\left \langle f\psi, (\lambda_1-J) f\psi \right \rangle=\sum_{e \in E, e=(v,w)} a_e \psi_v \psi_w(f_v-f_w)^2
\end{equation}
and the sum is over all edges, and $e=(v,w)$ means that $e$ connects $v$ and $w$.
\end{lem}

Christiansen, Simon, and Zinchenko use \eqref{eq:GroundState} and its analogue on the tree, together with the variational characterization of the $\sup$ of the spectrum, to obtain a comparison inequality between two different periodic Jacobi matrices arising from the same graph. The purpose of this short section is to note that it is also possible to combine \eqref{eq:GroundState} with the min-max principle to compare eigenvalue gaps for two Jacobi matrices on $\mathcal{G}$. Explicitly, we have
\begin{equation} \label{eq:MinMaxGround}
\lambda_1(J)-\lambda_k(J) =
\max_{\underset{\textrm{ dim }W=n-k+1}{W \subseteq L^2(\mathcal{G}) }}
\min_{f \in W} \sum_{e \in E, e=(v,w)} a_e \psi_v \psi_w \frac{(f_v-f_w)^2}{\|f\psi\|^2},
\end{equation}
and therefore
\begin{thm} \label{thm:CSZForFinite}
Let $\mathcal{G}$ be a finite connected graph and let $J$ and $\widetilde{J}$ be two Jacobi matrices defined on $\mathcal{G}$ with parameters $\{a_e, b_v\}_{e \in E, v \in V}$ and $\{\widetilde{a}_e,\widetilde{b}_v\}$, respectively. Let $\psi$ be the normalized Perron eigenfunction for $J$ and $\widetilde{\psi}$ the normalized Perron eigenfunction for $\widetilde{J}$ and let
\begin{equation} \nonumber
S=\max_{e=(v,w)\in E} \frac{a_e \psi_v\psi_w}{\widetilde{a}_e \widetilde{\psi}_v \widetilde{\psi}_w},\quad I=\min_{v\in V} \frac{\psi_v^2}{\widetilde{\psi}_v^2}.
\end{equation}
Then for any $k=1,\ldots,n$
\begin{equation} \label{eq:Comparison}
\lambda_1(J)-\lambda_k(J) \leq \frac{S}{I}\left(\lambda_1\left(\widetilde{J}\right)-\lambda_k\left(\widetilde{J}\right) \right).
\end{equation}
In particular, if $\psi=c\widetilde{\psi}$ for some constant $c >0$, then
\begin{equation} \label{eq:Comparison1}
\min_e \frac{a_e}{\widetilde{a}_e}\left(\lambda_1\left(\widetilde{J}\right)-\lambda_k\left(\widetilde{J}\right) \right)\leq \lambda_1(J)-\lambda_k(J)\leq \max_e \frac{a_e}{\widetilde{a}_e}\left(\lambda_1\left(\widetilde{J}\right)-\lambda_k\left(\widetilde{J}\right) \right)
\end{equation}
\end{thm}

A particular case of $\psi=\widetilde{\psi}$ occurs when $\psi$ and $\widetilde{\psi}$ are constant on $\mathcal{G}$. This holds for adjacency matrices on regular graphs, and, more generally, is guaranteed by the following condition.

\begin{lem}\label{lem-top-eigenvector}
	Let $J$ be a Jacobi matrix on a finite graph $\mathcal{G}$.
	Suppose that there exists a constant $C$
	such that for every vertex $v \in V(\mathcal{G})$,
	\begin{equation} \label{eq:comparison_condition}
		b_v +2\sum_{e \textrm{ is a loop at } v}a_e+
		\sum_{u\sim v}\sum_{e \in E, e=(v,u)} a_e
		= C.
	\end{equation}
	Then the Perron eigenvector
	of $J$ is the constant non-zero vector.
	
In particular, if $\mathcal{G}$ is a finite cover of a graph with a single vertex and $J$ is a lift of a Jacobi matrix defined on the single-vertex graph then the claim holds.
\end{lem}

\begin{proof}
	By the proof of the Perron-Frobenius Theorem (see \cite{perron}),
	the Perron eigenvector is the vector
	with positive entries maximizing the functional
	\begin{equation*}
		L(z) =
		\min_{\substack{v \in V(\mathcal{G}) \\ z_v \neq 0}}
		\frac{(Jz)_v}{z_v}
		=
		\min_{\substack{v \in V(\mathcal{G}) \\ z_v \neq 0}}
		\frac{b_v z_v +2\sum_{e \textrm{ is a loop at } v}a_ez_v+
		\sum_{u\sim v}\sum_{e \in E e=(v,u)} a_e z_u	
		}{z_v}
	\end{equation*}
	First observe that if $z$ is a constant non-zero vector,
	then
	\begin{equation*}
		L(z) = C.
	\end{equation*}
	On the other hand, if $z$ is a non-constant vector with positive entries,
	then by considering a vertex $w$ such that $z_w = \max_{v \in V(\mathcal{G})} z_v$ and that has a neighbor where $z$ obtains a value strictly smaller then $z_w$, we see that
	\begin{equation*}
		L(z) \leq C.
	\end{equation*}
	Therefore the constant vector indeed maximizes $L$,
	and must be the Perron eigenvector.
\end{proof}

\begin{rem}
It is not hard to see, e.g.\ by comparing the adjacency matrix on a regular graph to another matrix satisfying \eqref{eq:comparison_condition}, that the inequalities are not tight, especially when one of the edge weights approaches zero.
\end{rem}

\appendix
\section{Appendix}
\subsection{Proof of Theorem \ref{thm-greenberg-jacobi}} \label{GreenbergProof}

We need the following two simple lemmas.

\begin{lem}\label{lem-lift-spec-radius}
	Let $J_0$ be a Jacobi matrix
	on a finite graph $\mathcal G_0$,
	and $J$ its lift to a finite cover $\mathcal G$.
	Then $\lambda_1(J_0) = \lambda_1(J)$.
\end{lem}
\begin{rem}
	If $J_0$ has only non-negative entries,
	then by Perron-Frobenius $\lambda_1$ equals
	the spectral radius. 
	So in that case, the above
	lemma also shows equality of spectral radii.
\end{rem}
\begin{proof}
	Let $m \in \R$ be large enough so that
	$J_0 + m I, J + m I$
	have only non-negative entries.
	Let $f_0$ be the Perron-Frobenius eigenvector of $J_0 + m I$,
	namely $f_0$ has non-negative entries and is
	associated to the eigenvalue $\lambda_1(J_0 + m I) = \lambda_1(J_0) + m$.
	Denote by $p:J \to J_0$ the covering map.
	Then $f := f_0 \circ p$ is an
	eigenvector of $J + m I$ associated to the 
	eigenvalue $\lambda_1(J_0) + m$.
	Since $f$ has non-negative entries,
	by the uniqueness in the Perron-Frobenius theorem,
	it must be the Perron-Frobenius eigenvector of $J + m I$,
	so it is associated to the eigenvalue 
	$\lambda_1(J+mI) = \lambda_1(J) + m$.
\end{proof}

\begin{lem}\label{lem-measuresupport}
	Let $\mu$ be a finite measure,
	and denote by $[a,b]$ the convex hull of its topological support.
	Then
	\begin{equation*}
		\limsup_{\ell \to \infty}
		\left( \int_\R x^\ell d\mu(x) \right)^{1/\ell}
		= \max\{|a|, |b|\}.
	\end{equation*}
	Moreover, if $a \geq 0$, then the $\limsup$ above
	can be replaced by $\lim$.
\end{lem}
\begin{proof}
	Denote $M = \max\{|a|, |b|\}$.
	On the one hand,
	\begin{equation*}
		\left( \int_\R x^\ell d\mu(x) \right)^{1/\ell}
		\leq
		\left( \int_\R M^\ell d\mu(x) \right)^{1/\ell}
		= M \cdot \mu(\R)^{1/\ell}
		\underset{\ell \to \infty}{\longrightarrow}
		M.
	\end{equation*}
	And on the other hand, for every $\delta>0$,
	\begin{align*}
		\limsup_{\ell \to \infty}
		\left( \int_\R
			x^\ell d\mu(x) \right)^{1/\ell}
		&\geq
		\limsup_{\ell \to \infty}
		\left( \int_\R
			x^{2\ell} d\mu(x) \right)^{1/2\ell} \\
		& \geq
		\limsup_{\ell \to \infty}
		\left( \int_{M - \delta}^M
			x^{2\ell} d\mu(x) \right)^{1/2\ell} \\
		& \geq
		\limsup_{\ell \to \infty}
		\left( \int_{M - \delta}^M
			(M - \delta)^{2\ell} d\mu(x) \right)^{1/2\ell} \\
		& =
		(M - \delta) \cdot
		\limsup_{\ell \to \infty}
		\mu([M - \delta, M])^{1/2 \ell} \\
		&= M - \delta.
	\end{align*}
	The second inequality above holds because
	the integrand is non-negative.
	In the special case $a \geq 0$, this is true
	even without taking a subsequence.

\end{proof}

\begin{proof}[Proof of Theorem \ref{thm-greenberg-jacobi}]
	First we note that it suffices to prove only for
	the special case where
	$J_0$ and $J_n$ are all positive definite.
	By Lemma \ref{lem-lift-spec-radius},
	there exists a constant $m \in \R$
	such that $J_0 + m I$
	and $J_n + m I$
	are all positive definite.
	Because $J_n + m I$ and $J_T + m I$
	are lifts of $J_0 + m I$,
	and
	\begin{align*}
		\lambda_2(J_n + m I) &=
		\lambda_2(J_n) + m, \\
		\sup \sigma(J_T + m I)
		&= \sup\sigma(J_T) + m,
	\end{align*}
	we see that the positive definite case implies the
	general case.

	Denote by $\nu_{n}$ the normalized
	eigenvalue counting measure of $J_n$,
	and by $F_{n}(x) = \nu_n((-\infty, x])$
	its cumulative function.
	Suppose, for the sake of contradiction, that
	\begin{equation*}
		\liminf_{n\to\infty}
		\lambda_2(J_n)
		<
		\sup\sigma(J_T).
	\end{equation*}
	So there exists
	some $\varepsilon > 0$
	and a subsequence
	(which we take to be the original sequence
	for convenience)
	such that
	\begin{equation*}
		F_{n}
		\left(
		\sup\sigma(J_T) - \varepsilon
		\right)
		\underset{n\to\infty}{\longrightarrow} 1.
	\end{equation*}
	All the measures $d\nu_{n}$
	are supported on a common bounded set
	(see Lemma \ref{lem-lift-spec-radius}),
	contained in the positive reals
	by the positive definite assumption.
	By the Helly-Bray Theorem
	there exists a subsequential weak limit $d\nu
	= \lim d\nu_{n_k}$,
	also supported on the positive reals.
	Moreover, if $F$ denotes the cumulative function of $\nu$,
	then $F(x) = \lim_{k\to\infty}F_{n_k}(x)$
	for every continuity point $x$ of $F$.
	Since the set of discontinuity points of $F$ is countable,
	there exists $\varepsilon' \in (0,\varepsilon)$
	such that $F(\sup\sigma(J_T)) - \varepsilon') = 1$.
	It follows from Lemma \ref{lem-measuresupport} that
	\begin{equation}\label{ineq-tobecontradicted}
		\limsup_{\ell\to\infty}
		\left( \int_\R a^\ell d\nu(a) \right)^{1/\ell}
		\leq
		\sup\sigma(J_T) -\varepsilon'.
	\end{equation}
	On the other hand,
	\begin{equation*}
		\left( \int a^\ell d\nu(a) \right)^{1/\ell}
		=
		\lim_{k\to\infty}
		\left( \int a^\ell d\nu_{n_k}(a) \right)^{1/\ell}
		=
		\lim_{k\to\infty}
		\left(
			\text{tr}(J_{n_k}^\ell)
		\right)^{1/\ell},
	\end{equation*}
	where $\text{tr}$ is the normalized trace.
	That is,
	\begin{equation}\label{eq-tracemidcover}
		\left(
			\text{tr}(J_{n_k}^\ell)
		\right)^{1/\ell}
		=
		\left(
		\frac{1}{ \left| V \left( \mathcal{G}_{n_k} \right) \right|}
		\sum_{v \in V \left( \mathcal{G}_{n_k} \right)}
		\langle \delta_v , J_{n_k}^\ell \delta_v \rangle
		\right)^{1/\ell}
	\end{equation}
	We will bound this expression from below by a term
	independent of $k$, converging as $\ell \to \infty$
	to $\sup \sigma(J_T)$,
	thus contradicting \eqref{ineq-tobecontradicted}.

	Recall that there are covering maps
	\begin{center}
	\begin{tikzpicture}
		\node (Basis) at (0,0) {};
		\node[above=of Basis] (T) {$T$};
		\node[right=of Basis] (Gnk) {$\mathcal{G}_{n_k}$};
		\node[below=of Basis] (G) {$\mathcal{G}$};
		\draw[->] (T)--(Gnk) node [midway,above] {$q_k$};
		\draw[->] (T)--(G) node [midway,left] {$p$};
		\draw[->] (Gnk)--(G) node [midway,below] {$p_k$};
	\end{tikzpicture}
	\end{center}
	Because all the Jacobi parameters 
	may be assumed to be non-negative,
	and the number of closed paths
	in $\mathcal{G}_{n_k}$ is greater than in $T$,
	every $v \in V \left( \mathcal{G}_{n_k} \right)$
	and $u \in q_k^{-1}(v)$
	obey the inequality
	\begin{equation*}
		\langle \delta_v , J_{n_k}^\ell \delta_v \rangle
		\geq
		\langle \delta_u , J_{T}^\ell \delta_u \rangle.
	\end{equation*}
	Plug this inequality back into
	equation \eqref{eq-tracemidcover}
	and group together elements coming from
	the same fiber in $\mathcal{G}_{n_k}$ over $\mathcal{G}$.
	\begin{equation*}
		\left(
			\text{tr}(J_{n_k}^\ell)
		\right)^{1/\ell}
		\geq
		\left(
		\frac{1}{ \left| V \left( \mathcal{G} \right) \right|}
		\sum_{w \in V(\mathcal{G})}
		\langle \delta_{u(w)}, J_T^\ell \delta_{u(w)} \rangle
		\right)^{1/\ell}
	\end{equation*}
	where $u(w)$ is an arbitrary element in $p^{-1}(\{w\})$.
	For any fixed $u\in V(T)$, by the non-negativity
	of the summands, it follows that
	\begin{equation*}
		\left(
		\text{tr}(J_{n_k}^\ell)
		\right)^{1/\ell}
		\geq
		\left(
		\frac{1}{|V(\mathcal{G})|}
		\left<
			\delta_u, J_T^\ell \delta_u
		\right>
		\right)^{1/\ell}.
	\end{equation*}
And since $\left< \delta_u, J_T^\ell \delta_u \right> ^{1/\ell}\rightarrow \sup \sigma(J_T)$ as $\ell$ goes to infinity, for any $u$, we are done.
\end{proof}


\end{document}